\documentclass{amsart}
\usepackage{verbatim,url,psfrag,graphicx}
\usepackage[dvips]{geometry}

\newtheorem{theorem}{Theorem}[section]
\newtheorem{lemma}[theorem]{Lemma}
\newtheorem{proposition}[theorem]{Proposition}

\newtheorem{conjecture}[theorem]{Conjecture}
\newtheorem{corollary}[theorem]{Corollary}

\title{Large cliques in graphs with high chromatic number}
\author{Csaba Bir\'o}
\date{May 2010}
\address{University of Louisville, Louisville, KY 40292, USA}
\email{csaba.biro@louisville.edu}

\begin{document}

\begin{abstract}
We study graphs whose chromatic number is close to the order of the graph (the
number of vertices). Both when the chromatic number is a constant multiple of
the order and when the difference of the chromatic number and the order is a
small fixed number, large cliques are forced. We study the latter situation,
and we give quantitative results how large the clique number of these graphs
have to be.  Some related questions are discussed and conjectures are posed.

Please note that the results of this article were significantly generalized.
Therefore this paper will never be published in a journal.
See instead \cite{Bir-Fur-Jah-11-u} for the more general results.
\end{abstract}

\maketitle

\section{Introduction}

It is well known that for all positive integer $c$, triangle free graphs with
chromatic number $c$ exist. The number of vertices required to obtain such
graphs grows quadratically with $c$. So when the chromatic number is close to
the order of the graph, large cliques are forced. But just how large these
cliques have to be? To make these questions precise, we introduce the following
definition.
\[
Q(n,c)=\min\{\omega(G):|V(G)|=n\text{ and }\chi(G)=c\}
\]

Suppose that $0<r<1$. One can ask about the behavior of the sequence
$Q(n,\lceil rn \rceil)$, in particular, if it converges to infinity. The answer
is positive for all $r$, and it is fairly easy to see: it follows from the
Ramsey Theory result that a graph with no cliques of size $t$ and $n$ vertices
has independence number $\Omega(n^{1/t})$, which is a straightforward
consequence of the classical Erd{\H o}s--Szekeres \cite{Erd-Sze-35} bound for
the Ramsey numbers.


In this paper we study the behavior $Q(n,n-k)$ if $k$ is small relative to $n$.
In the following, if not otherwise indicated, all the statements about
$Q(n,n-k)$ is to be understood with the condition ``if $n$ is sufficiently
large''. It is immediate from the definition that $Q(n,n)=n$, $Q(n,n-1)=n-1$,
$Q(n,n-2)=n-3$ (remove the edges of $C_5$ from $K_n$ to see the last result).
We will determine $Q(n,n-k)$ for $k\leq 6$.

Throughout the paper we will call single vertex color classes in a valid
coloring of a graph ``singletons''. Similarly, we will use the term
``doubleton'' for color classes containing two vertices.

\begin{proposition}
Let $k\geq 3$. If $n$ is large enough, then $n-2k+2\leq Q(n,n-k)\leq n-k-1$.
\end{proposition}

\begin{proof}
To see that $Q(n,n-k)\leq n-k-1$, remove the edges of $C_{2k+1}$ from $K_n$.
It remains to be shown that $Q(n,n-k)\geq n-2k+2$. Consider a graph $G$ with
$n$ vertices and $\chi(G)=n-k$. Consider an optimal coloring of the vertices of
$G$ with $n-k$ colors. If there are at least $n-2k+1$ singletons, it is easy to
see that there is a $K_{n-2k+2}$ subgraph, so we may assume that there are
$n-2k$ singletons, forming a clique $Q$. The remaining $2k$ vertices are in $k$
classes, so all of them are doubletons. In each of the doubleton classes, one
of the vertices is adjacent to every vertex in $Q$, otherwise $\chi(G)<n-k$.
Denote the set of these vertices by $S$. If $S$ is not an independent set, then
there is a $K_{n-2k+2}$ and we are done. If $S$ is an independent set, then
remove $S$ from $G$. The remaining graph $G'$ has $|V(G')|=n-k$ and
$\chi(G')\geq n-k-1$, so there is a clique of size $n-k-1\geq n-2k+2$ in $G'$
and in turn in $G$.
\end{proof}

\begin{corollary}
$Q(n,n-3)=n-4$
\end{corollary}

Based on these, the author (naively) made the following conjecture.
\begin{conjecture}\label{conj:main}
Let $k$ be a positive integer. If $n$ is large enough, then
\[
Q(n,n-k)=n-2k+\lceil k/2\rceil.
\]
\end{conjecture}

\emph{This conjecture was disproven in \cite{Bir-Fur-Jah-11-u}.}

Jahanbekam and West \cite{Jah-Wes} observed that $Q(n,n-k)$ is at most the
conjectured value whenever $n\geq\lceil 5k/2\rceil$ and they conjecture that
this threshold on $n$ is both sufficient and necessary for equality. Their
constructions are the following. If $k$ is even, let $r=k/2$ and let $G$ be the
complement of $rC_5$. If $k$ is odd, let $r=(k+1)/2$ and let $G$ be complement
of $(r-1)C_5+P_3$. Then adding a dominating vertex to $G$ increases $|V(G)|$,
$\chi(G)$ and $\omega(G)$ each by $1$.

Our main result is the following.
\begin{theorem}\label{thm:main}
Let $k\geq 5$. If $n$ is large enough, then $Q(n,n-k)\geq n-2k+3$.
\end{theorem}

\begin{corollary}
Conjecture~\ref{conj:main} is true for $k\leq 6$.
\end{corollary}

\section{Motivation and related research}

The original motivation of this research was an analogue problem for partially
ordered sets (posets).

A \emph{realizer} is a set of linear extensions of the
poset $P$, such that their intersection (as relations) is $P$. The minimum
cardinality of a realizer is the dimension of the poset, a central notion in
poset theory. The ``standard example'' $S_n$ is the poset formed by considering
the $1$-element subsets and the $n-1$ element subsets of a set of $n$ elements,
ordered by inclusion. It is well known that $\dim(S_n)=n$, but there are posets
of arbitrarily large dimensions without including even $S_3$ as a subposet.

Hiraguchi \cite{Hir-55} proved that the dimension does not exceed half of the
number of elements of the poset. Bogart and Trotter \cite{Bog-Tro-73} showed
that for large $n$, the only $n$-dimensional poset on $2n$ points is $S_n$. But
what happens if the dimension is slightly less than half the number of
elements, is not known. We conjecture the following.

\begin{conjecture}
For every $t<1$, but sufficiently close to $1$
there is a $c>0$ such that if a poset has $2n$ points, and its dimension is at
least $tn$, then it contains a standard example of dimension $cn$.
\end{conjecture}

It is frequently noted that poset problems can be translated to graph theory
problems and vice versa by changing chromatic numbers of graphs to dimension of
posets, and cliques in graphs to standard examples in posets. Note that the
above conjecture would translate to the following statement: For every $t<1$,
but sufficiently close to $1$ there is a $c>0$ such that if a graph has $n$
points, and its chromatic number is at least $tn$, then it contains a clique of
$cn$ points. This graph version is trivial for all $t>1/2$.

There are some potentially interesting relation of these questions to list
coloring of graphs. Reed and Sudakov \cite{Ree-Sud-05} showed that if the
chromatic number is at least about $3/5$ of the number of vertices, then the
list chromatic number and the chromatic number are equal. List coloring also
comes up explicitly and implicitly in the proof of Theorem \ref{thm:main},
further suggesting that other list coloring problems may be related. Another
article with potential connections is \cite{Far-Mol-Ree-05}.

\section{Proof of Theorem~\ref{thm:main}}

Consider a graph $G$ with $|V(G)|=n$, $\chi(G)=n-k$, and consider an optimal
coloring $\gamma:V(G)\to\{1,\ldots,n-k\}$ of $G$. If the number of singletons is at
least $n-2k+2$ it is easy to see that there is a $K_{n-2k+3}$ subgraph. So we
may assume that the number of singletons is either $n-2k+1$ or $n-2k$. In
either case, the singletons form a clique, call it $Q$.

If $|Q|=n-2k+1$, then the remaining $2k-1$ vertices form $k-1$ color classes:
that is $k-2$ doubletons and one $3$-element class. In each of these classes,
at least one of the vertices is adjacent to every vertex in $Q$, otherwise
$\chi(G)\leq n-k$. Call these vertices $S$. If $S$ is not an independent set,
then there is a $K_{n-2k+3}$; otherwise remove $S$ to form the graph $G_1$. Now
$|V(G_1)|=n-k+1$, $\chi(G_1)\geq n-k-1$ so $\omega(G_1)\geq n-k-2\geq n-2k+3$,
providing the clique that we seek.

Hence we may assume that $|Q|=n-2k$, and the remaining color classes are $k$
doubletons. It is still true that every color class contains a vertex that is
adjacent to every vertex in $Q$, and we will keep calling this set $S$. But now
all we can deduce immediately that if there is a triangle in $S$, we are done.
So we may assume $S$ contains no triangles.

Since the diagonal Ramsey number $R(3,3)=6$, it follows that if $S$ is triangle
free and $k=|S|\geq 6$, then $S$ contains a $3$-element independent set.
Removing this from $G$ we get a graph $G_2$ with $|V(G_2)|=n-3$, and
$\chi(G_2)\geq n-k-1$. The difference is $k-2$, so an induction argument
implies that $G_2$ contains a clique of size $n-3-2(k-2)+\lceil
(k-2)/2\rceil=n-2k+\lceil k/2\rceil$.


So we may assume that $k=|S|=5$ and that $S$ does not contain any triangle or
$3$-element independent set, i.e.~it induces a $C_5$. Denote the elements
of $S$ by $v_1,\ldots,v_5$, and the other elements in the doubleton classes
by $u_1,\ldots,u_5$ such that $\{u_i,v_i\}$ are the doubletons for
$i=1,2,\ldots,5$. (Figure~\ref{fig:T}.)
\begin{figure}
\psfrag{u1}{$u_1$}
\psfrag{u2}{$u_2$}
\psfrag{u3}{$u_3$}
\psfrag{u4}{$u_4$}
\psfrag{u5}{$u_5$}
\psfrag{v1}{$v_1$}
\psfrag{v2}{$v_2$}
\psfrag{v3}{$v_3$}
\psfrag{v4}{$v_4$}
\psfrag{v5}{$v_5$}
\psfrag{Q}{$Q$}
\includegraphics[scale=0.5]{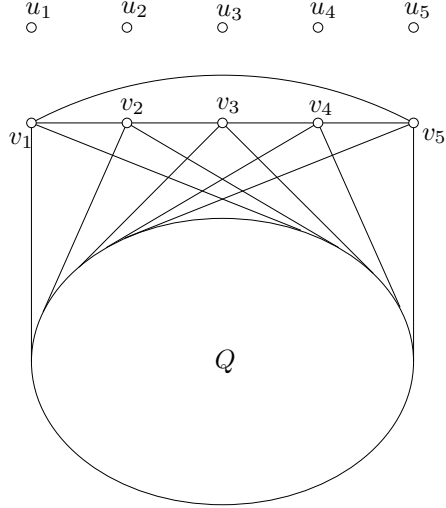}
\caption{The structure of $G$ when $|S|=5$}\label{fig:T}
\end{figure}

Let $U$ be the graph induced by $u_1,\ldots,u_5$, and 
$T$ be the graph induced by $u_1,\ldots,u_5,v_1,\ldots,v_5$. Observe that
the complement of $T$ is triangle free. This is because a triangle in
$\overline{T}$ is a 3-element independent set in $T$. Removing this from $G$ we
would get a graph $G_3$ with $|V(G_3)|=n-3$, $\chi(G_3)\geq n-6$, and so
$\omega(G_3)\geq n-7$.

Without loss of generality we will assume that $\{\gamma(v):v\in
Q\}=\{6,7,\ldots,n-5\}$, and $\gamma(v_i)=\gamma(u_i)=i$ for $i=1,\ldots,5$. If
we can recolor the vertices of $T$ with fewer than 5 colors, we get a
contradiction with the chromatic number of $G$. The set of colors that are
available for recoloring is easily identifiable. Let $E=\{1,\ldots,4\}$ and let
$L_i=\{\gamma(v): v\in Q\text{ and }u_i\not\sim v\}$. Our goal is to list color
$T$ with the list $E$ assigned to the vertices of $S$ and the lists $E\cup L_i$
assigned to $u_i$ for $i=1,\ldots,5$. This way, we never use the color $5$, so
we achieve the desired contradiction.

Note that $L_i\neq\emptyset$ for any $i$. Indeed, if $L_k=\emptyset$, then
$|S\cup \{u_k\}|=6$ so this set contains either a triangle or a $3$-element
independent set, and previously we have seen that this is sufficient to finish
the proof.

Let $L=\cup_{i=1}^5 L_i$ and $l=|L|$. We have seen that $l>0$. We will analyze
$5$ cases, depending on the magnitude of $l$. In each case, we will attempt to
list color $T$ with the given lists, and if it is impossible, we will show the
existence of a clique of size $n-7$. In each case, we will use the colors
$1,2,3$ to color $S$; the colors $1$ and $2$ will be used twice and the color
$3$ will be used once. However, we will first focus on coloring the vertices
of $U$ with the following self-imposed restrictions:
\begin{itemize}
\item We will not use the colors $1$ and $2$.
\item We will use color $3$ at most once.
\end{itemize}
We can use the color $4$ completely freely,
because we won't need it for $S$, and of course we can use the colors from
$L_i$ to color $u_i$. Once $U$ is colored, we find the vertex $u_k$ (if exists)
for which color $3$ was used (otherwise pick $u_k$ arbitrarily). Then we assign
the color $3$ to $v_k$ and the colors $1$ and $2$ accordingly to the other four
vertices in $S$ to create proper coloring.

In many cases it will be useful to illustrate the coloring of $U$ with a
bipartite graph $B$. One of the partite set is $V(U)$, the other is $L$ and the
vertex $u_k$ is adjacent to the color $c$ if $c\in L_k$. A matching in $B$
corresponds to a valid (possibly partial) coloring of $U$. This will be
especially useful in the cases when $l$ is large, so a large matching can be
found. We will make use of Hall's Theorem, and a set of vertices that violates
the condition of the theorem will be called a Hall-block. In particular, if a
set of vertices $X$ has the property that $|N(X)|<|X|$, and $|X|=i$ and
$|N(X)|=j$, then we will call $X$ an ``$i$-$j$'' block.

\subsection{If $l\geq 4$}\label{case1} Suppose that $L=\{a,b,c,d\}$ and let $B$ the
bipartite graph defined above. If $|N(\{a,b,c\})|\leq 2$ and
$|N(\{b,c,d\})|\leq 2$, then $|N(\{a,b,c,d\})|\leq 4$, so there is vertex with
an empty list, contradiction. So there is a $3$-element subset $L'$ of $L$ with
$|N(L')|\geq 3$. Without loss of generality, $L'=\{a,b,c\}$.

If there is matching that covers $L'$, we are done: we can use the
corresponding coloring for $3$ vertices in $N(L')$ and the colors $3$ and $4$
for the remaining two vertices. If not, then there is a Hall-block in the
subgraph of $B$ induced by $L'$ and $N(L')$. The only possible block is 2-1
block, because 2-0 or 1-0 blocks would mean $l<4$. Without loss of generality,
$|N(\{a,b\})|=1$. We again have to separate
cases by the size of $N(L')$.

\subsubsection{If $|N(L')|=3$} In this case $B$ must contain the following
subgraph. (The vertices $u_1,\ldots,u_5$ may be permuted. This is also true for 
the later subcases of Case~\ref{case1} and for similar claims in later cases.)
\begin{center}
\psfrag{u1}{$u_1$}
\psfrag{u2}{$u_2$}
\psfrag{u3}{$u_3$}
\psfrag{u4}{$u_4$}
\psfrag{u5}{$u_5$}
\psfrag{a}{$a$}
\psfrag{b}{$b$}
\psfrag{c}{$c$}
\psfrag{d}{$d$}
\includegraphics[scale=0.5]{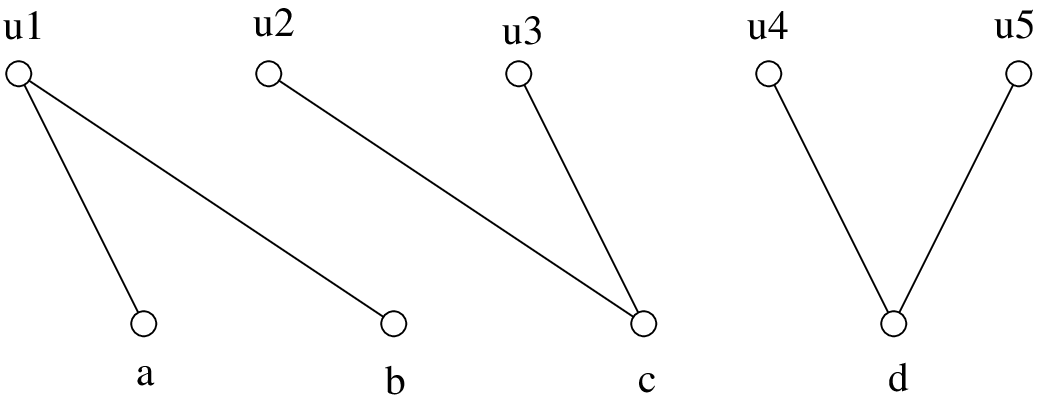}
\end{center}
This clearly contains a matching of size $3$, and the remaining vertices can be
colored with colors $3$ and $4$.

\subsubsection{If $|N(L')|=4$} In this case $B$ must contain the following
subgraph.
\begin{center}
\psfrag{u1}{$u_1$}
\psfrag{u2}{$u_2$}
\psfrag{u3}{$u_3$}
\psfrag{u4}{$u_4$}
\psfrag{u5}{$u_5$}
\psfrag{a}{$a$}
\psfrag{b}{$b$}
\psfrag{c}{$c$}
\psfrag{d}{$d$}
\includegraphics[scale=0.5]{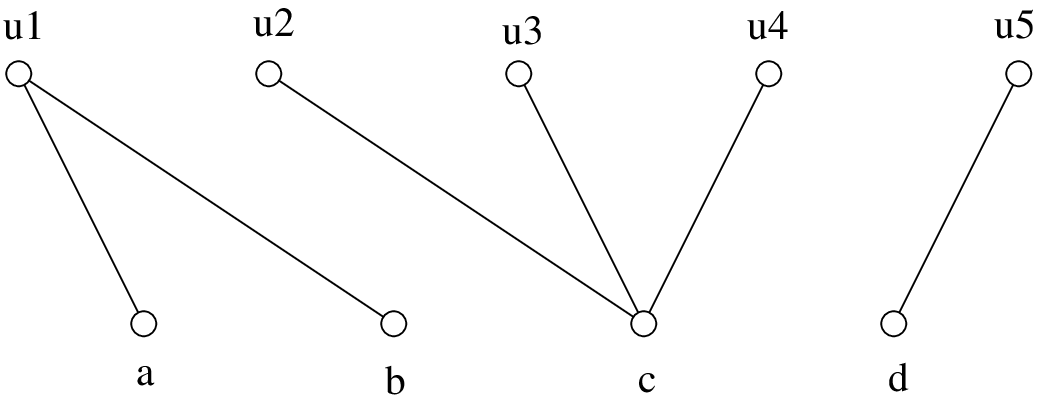}
\end{center}
This, again contains a matching of size $3$.

\subsubsection{If $|N(L')|=5$}\label{case1c} In this case $B$ must contain one of the
following subgraphs.
\begin{center}
\begin{tabular}{cc}
\psfrag{u1}{$u_1$}
\psfrag{u2}{$u_2$}
\psfrag{u3}{$u_3$}
\psfrag{u4}{$u_4$}
\psfrag{u5}{$u_5$}
\psfrag{a}{$a$}
\psfrag{b}{$b$}
\psfrag{c}{$c$}
\psfrag{d}{$d$}
\includegraphics[scale=0.5,bb=0 0 350 1]{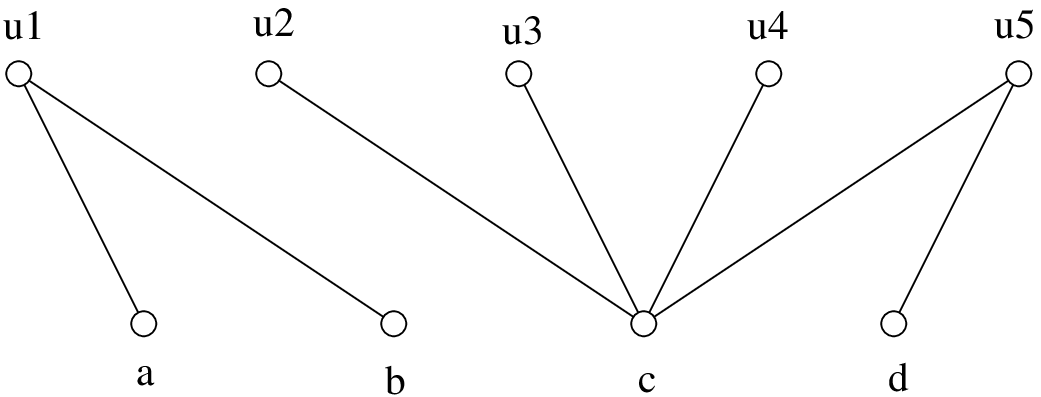}
&
\psfrag{u1}{$u_1$}
\psfrag{u2}{$u_2$}
\psfrag{u3}{$u_3$}
\psfrag{u4}{$u_4$}
\psfrag{u5}{$u_5$}
\psfrag{a}{$a$}
\psfrag{b}{$b$}
\psfrag{c}{$c$}
\psfrag{d}{$d$}
\includegraphics[scale=0.5]{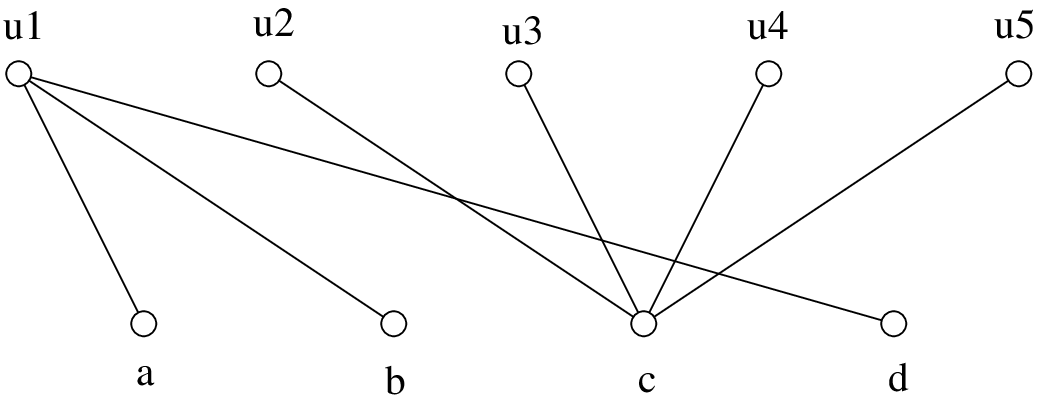}
\end{tabular}
\end{center}
The one on the left contains a matching of size $3$. The one on the right shows
a pattern that we will see again later. If there is a non-edge between any two
of $u_2,u_3,u_4,u_5$, then those two can be colored with $c$, $u_1$ can be
colored with any of $a$, $b$ or $d$, and the rest two vertices can use $3$ and
$4$. If $u_2,u_3,u_4,u_5$ induce a clique, then let $v$ be the vertex of $Q$
with color $c$; then $(Q\setminus\{v\})\cup\{u_2,u_3,u_4,u_5\}$ induce a clique
of size $n-7$.

\subsection{If $l=3$} Suppose $L=\{a,b,c\}$, and let $B$ be the graph as
above. Unless there is a 2-1 block in $B$, there is a complete matching from
the colors, and we are done. If there is a 2-1 block, then $B$ contains the
following subgraph.
\begin{center}
\psfrag{u1}{$u_1$}
\psfrag{u2}{$u_2$}
\psfrag{u3}{$u_3$}
\psfrag{u4}{$u_4$}
\psfrag{u5}{$u_5$}
\psfrag{a}{$a$}
\psfrag{b}{$b$}
\psfrag{c}{$c$}
\includegraphics[scale=0.5]{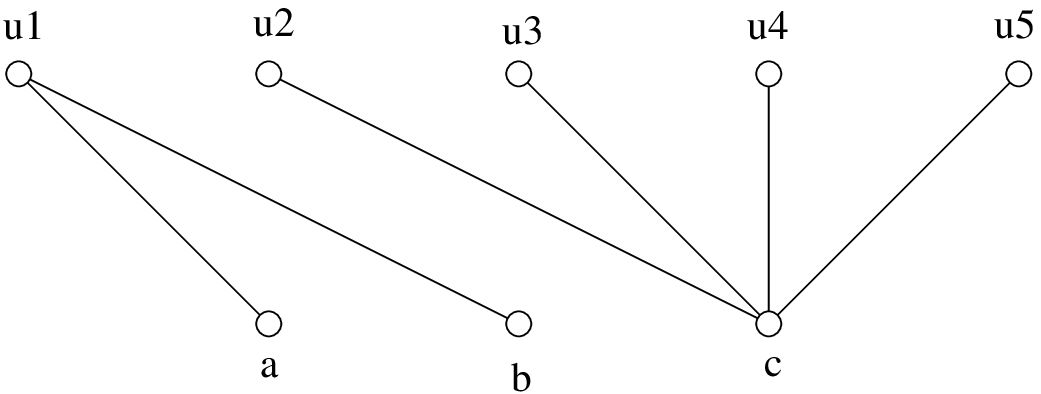}
\end{center}
This situation can be dealt with in exactly the same way as the second subcase
of Case~\ref{case1c}. The only difference is that there are only two colors for $u_1$,
but we really only need one.

\subsection{If $l=2$} Suppose $L=\{a,b\}$. In this case, each list is one
of $\{a\}$, $\{b\}$, and $\{a,b\}$. Let $A=\{u_i:L_i=\{a\}\}$,
$B=\{u_i:L_i=\{b\}\}$, and $M=\{u_i:L_i=\{a,b\}\}$. We claim that we may assume
that all of $A$, $B$ and $M$ induce a clique.

To see this, suppose that $A$ does not induce a clique. Then there is a nonedge
between two vertices of $A$. We can color these with the color $a$, and look at
the remaining three vertices. At least one of those must contain $b$ in its
list, otherwise $l<2$. So find one that has $b$ in its list, color that
with $b$, and the other two takes the colors $3$ and $4$. Similarly, we can see
that $B$ induces a clique. It remains to be seen that $M$ induces a clique.
Suppose not, and there is a nonedge between vertices of $M$. Without loss of
generality, these vertices are $u_1$ and $u_2$. The set $L_3\cup L_4\cup L_5$
is a nonempty subset of $L$, so it either contains $a$ or $b$, without loss of
generality $a\in L_3\cup L_4\cup L_5$. Again, we may assume $a\in L_3$.
Then use the following coloring:
\begin{align*}
u_1,u_2&\to b\\
u_3&\to a\\
u_4&\to 3\\
u_5&\to 4
\end{align*}

On the other had, it is clear that we may assume that $\{u_1,\ldots,u_5\}$ does
not induce a clique. If it does, and $v$ and $w$ are the vertices of $Q$ with $\gamma(v)=a$ and $\gamma(w)=b$, then
$(Q\setminus\{v,w\})\cup\{u_1,\ldots,u_5\}$ induces a clique of size $n-7$.

Now we will show that we may assume that $A\cup M$ and $B\cup M$ are cliques.
Without loss of generality $A\cup M$ is not a clique, so there is a nonedge, and
it can only be between a vertex whose list is $\{a\}$ (say $u_1$) and another
vertex whose list is $\{a,b\}$ (say $u_2$). If $b\in L_3\cup L_4\cup L_5$ (say
$b\in L_3$), then use the coloring:
\begin{align*}
u_1,u_2&\to a\\
u_3&\to b\\
u_4&\to 3\\
u_5&\to 4
\end{align*}
If $b\not\in L_3\cup L_4\cup L_5$, then $L_1=L_3=L_4=L_5=\{a\}$. If
$\{u_1,u_3,u_4,u_5\}$ induces a clique, we are done, otherwise, say
$u_1\not\sim u_3$, use the following coloring:
\begin{align*}
u_1,u_3&\to a\\
u_2&\to b\\
u_4&\to 3\\
u_5&\to 4
\end{align*}

So we have shown that there is a nonedge in $\{u_1,\ldots,u_5\}$, say
$u_1\not\sim u_2$, and $L(u_1)=\{a\}$ and $L(u_2)=\{b\}$. If
$\{a,b\}\not\subseteq L_3\cup L_4\cup L_5$ then either four of lists is $\{a\}$
and the fifth is $\{b\}$ or vice versa; both cases can be handled similarly as
before. So we have $\{a,b\}\subseteq L_3\cup L_4\cup L_5$. Then there is a list
among $L_3,L_4,L_5$ that contains $a$, say $L_3$, and another that contains
$b$, say $L_4$, and then the following coloring works:
\begin{align*}
u_1,u_2&\to 4\\
u_3&\to a\\
u_4&\to b\\
u_5&\to 3
\end{align*}

\subsection{If $l=1$}
Suppose that $L=\{a\}$ and $\overline{U}$ contains a matching of size $2$, say
$\{u_1,u_2\}$ and $\{u_3,u_4\}$. Then we can use the following coloring:
\begin{align*}
u_1,u_2&\to a\\
u_3,u_4&\to 4\\
u_5&\to 3
\end{align*}
Otherwise $\overline{U}$ is a star (recall that it is triangle free). Without
loss of generality, the center of the star is $u_1$. Let $v$ be the vertex of
$Q$ for which $\gamma(v)=a$.  Then $(Q\setminus \{v\})\cup\{u_2,u_3,u_4,u_5\}$
is a clique of size $n-7$.\qed

\section{Acknowledgements}
The author wish to thank to Joshua Cooper, Vladimir Nikiforov and L\'aszl\'o
Sz\'ekely for their comments and ideas on these questions.

\bibliography{bib,extra}
\bibliographystyle{amsplain}

\end{document}